\documentclass[a4paper,12pt,reqno]{amsart}
\usepackage{a4wide}
\usepackage{amsmath}
\usepackage{amssymb}
\usepackage{amsthm}
\usepackage{latexsym}
\usepackage{relsize}
\usepackage[english]{babel}
\usepackage{relsize}

\newtheorem{prop}[subsection]{Proposition}

\newtheorem{teor}[subsection]{Theorem}
\newtheorem{lema}[subsection]{Lemma}
\newtheorem{cor} [subsection]{Corollary}
\theoremstyle{definition}
\newtheorem{dfn} [subsection]{Definition}
\theoremstyle{remark}
\newtheorem{obs} [subsection]{Remark}
\newtheorem{exm} [subsection]{Example}

\newcommand{\Zng}{$\mathbb Z^n$-graded $S$-module}

\def\sdepth{\operatorname{sdepth}}
\def\qdepth{\operatorname{hdepth}}
\def\hdepth{\operatorname{hdepth}}
\def\depth{\operatorname{depth}}

\def\deg{\operatorname{deg}}
\def\lcm{\operatorname{lcm}}

\def\PP{\operatorname{P}}

\def\po#1#2{(#1)_#2}
\def\de{\delta}

\numberwithin{equation}{section}

\begin{document}

\title{On the Hilbert depth of monomial ideals}
\author{Silviu B\u al\u anescu$^1$, Mircea Cimpoea\c s$^2$ and Christian Krattenthaler$^3$}
\date{}





\footnotetext[1]{ \emph{Silviu B\u al\u anescu}, National University of Science and Technology Politehnica Bucharest, Faculty of
Applied Sciences, 
Bucharest, 060042, E-mail: silviu.balanescu@stud.fsa.upb.ro}
\footnotetext[2]{ \emph{Mircea Cimpoea\c s}, National University of Science and Technology Politehnica Bucharest, Faculty of
Applied Sciences, 
Bucharest, 060042, Romania and Simion Stoilow Institute of Mathematics, Research unit 5, P.O.Box 1-764,
Bucharest 014700, Romania, E-mail: mircea.cimpoeas@upb.ro,\;mircea.cimpoeas@imar.ro}
\footnotetext[3]{\emph{Christian Krattenthaler}, Fakult\"at f\"ur Mathematik,
Universit\"at Wien Oskar-Morgenstern-Platz 1 A-1090 Vienna, Austria, E-mail: Christian.Krattenthaler@univie.ac.at}

\begin{abstract}
Let $S=K[x_1,\ldots,x_n]$ be the ring of polynomials over a field $K$.
Given two monomial ideals $0\subset I\subsetneq J \subset S$, we present a new method to compute 
the Hilbert depth of $J/I$. As an application, we show that if $u\in S$ is a monomial regular of $S/I$, then 
$\qdepth(S/I)\geq \qdepth(S/(I,u))\geq \qdepth(S/I)-1.$

Also, we reprove the formula of the Hilbert depth of a squarefree Veronese ideal.
\end{abstract}


\maketitle

\section{Introduction}

Let $K$ be a field and $S=K[x_1,\ldots,x_n]$ the polynomial ring over $K$. We consider the standard grading on $S$.
Let $M$ be a finitely generated graded $S$-module. In \cite{uli}, Uliczka introduced a new invariant associated to $M$, called Hilbert depth. 
More precisely, the Hilbert depth of $M$, denoted by $\hdepth(M)$, is the maximal depth of a finitely generated graded $S$-module 
$N$ with the same Hilbert series as $M$; see \cite[Definition 3.1]{uli}. He also proved in \cite[Theorem~3.2]{uli} that
$$\hdepth(M)=\max\{r\;:\;(1-t)^rH_M(t)\text{ is non-negative }\}.$$
In \cite{bruns}, Bruns, Krattenthaler and Uliczka took a different approach regarding this invariant. 

But first, we need to recall the following definition:
Let $M$ be a \Zng. A \emph{Stanley decomposition} of $M$ is a direct sum $$\mathcal D: M = \bigoplus_{i=1}^rm_i K[Z_i],$$ as a 
$\mathbb Z^n$-graded $K$-vector space, where $m_i\in M$ is homogeneous with respect to $\mathbb Z^n$-grading, 
$Z_i\subset\{x_1,\ldots,x_n\}$ such that $m_i K[Z_i] = \{um_i:\; u\in K[Z_i] \}\subset M$ is a free $K[Z_i]$-submodule of $M$. 
We define $\sdepth(\mathcal D)=\min_{i=1,\ldots,r} |Z_i|$ and
$$\sdepth(M)=\max\{\sdepth(\mathcal D):\mathcal D\text{ is 
a Stanley decomposition of }M\}.$$
The number $\sdepth(M)$ is called the \emph{Stanley depth} of $M$.

Herzog, Vl\u adoiu and Zheng showed in \cite{hvz} that $\sdepth(M)$ can be computed in a finite number of steps if $M=I/J$, 
where $J\subset I\subset S$ are monomial ideals. In \cite{rin}, Rinaldo gave a computer implementation of this algorithm, 
in the computer algebra system $\mathtt{CoCoA}$ (cf. \cite{cocoa}). 
In \cite{apel}, Apel restated a conjecture firstly given by Stanley in \cite{stan}, namely that $$\sdepth(M)\geq\depth(M),$$ for any \Zng $\;M$. 
This conjecture proves to be false, in general, for 
$M=S/I$ and $M=J/I$, where $0\neq I\subsetneq J\subset S$ are monomial ideals; see \cite{duval}, but remains open for $M=I$.

Now, we return to the standard graded case and we let $M$ to be a finitely generated graded $S$-module. A Hilbert decomposition of
the Hilbert series $H_M(t)$ is a decomposition 
$$\mathcal H: H_M(t)=\sum_{i=1}^r \frac{t^{a_i}}{(1-t)^{b_i}},$$
where $a_i$ and $b_i$ are nonnegative integers. The Hilbert depth of $\mathcal H$ is 
$$\hdepth(\mathcal H)=\min\{b_i\;:\;1\leq i\leq r\}.$$
Bruns, Krattenthaler and Uliczka \cite{bruns} proved that 
$$\hdepth(M)=\max\{ \hdepth(\mathcal H)\;:\;\mathcal H\text{ is a Hilbert decomposition of }M.\}.$$
Also, they noted that if $M$ is a \Zng and $\mathcal D: M = \bigoplus_{i=1}^rm_i K[Z_i]$ is a Stanley decomposition of $M$, then
$\mathcal H: H_M(t)=\sum_{i=1}^r \frac{t^{a_i}}{(1-t)^{b_i}}$ is a Hilbert decomposition of $M$, regarded now as a graded $S$-module,
where $a_i=\deg(m_i)$ and $b_i=|Z_i|$ for all $1\leq i\leq r$. This implies immediately that
$$\hdepth(M)\geq \sdepth(M).$$
One would expect that it is easy to compute the Hilbert depth of a module, once its Hilbert function is known. 
But it turns out that even for the powers of the maximal ideal, the computation of the Hilbert depth leads to difficult
numerical computations; see \cite{maxim}. Another argument for studying this invariant is the fact that the Hilbert depth
of a finitely generated $\mathbb Z^n$-graded $S$-module $M$ is an upper bound for the Stanley depth of $M$, as we have seen above.

We note that there exists also a multigraded version of the Hilbert depth invariant, but it this beyond the scope of our article; 
see \cite{bruns} and \cite{ichim0}. Also, we mention that A. Popescu \cite{apop} give an algorithm which computes the Hilbert depth of a
graded $S$-module. For a friendly introduction into the thematic of Stanley depth and Hilbert depth and further details, 
we refer the reader to \cite{her}.


Given two squarefree monomial ideals $0\subset I\subsetneq J\subset S$, for all $0\leq j\leq n$, we let
$\alpha_j(J/I):=$ the number of squarefree monomials $u\in S$ such that $u\in J\setminus I$ and $\deg(u)=j$. Also, let
$$\beta_k^q(J/I)=\sum_{j=0}^k (-1)^{k-j} \binom{q-j}{k-j}\alpha_j(J/I)\text{ for all }0\leq k\leq q\leq n.$$
In Teorem \ref{teo1} we prove that 
$$\hdepth(J/I)=\max\{q\;:\;\beta_k^q(J/I)\geq 0\text{ for all }0\leq k\leq q\}.$$
In particular, in Corollary~\ref{p2}, we reprove the fact that $\hdepth(J/I)\geq \sdepth(J/I)$ for any monomial ideals 
$0\subset I\subsetneq J\subset S$, not necessarily squarefree. 

We emphasize again that, when we talk about the Hilbert depth of $J/I$ and the depth of $J/I$, 
we consider the standard graded structure of $J/I$, while, when we talk about the Stanley depth of $J/I$, we consider the
multigraded structure of $J/I$!

In Theorem \ref{tu} we show that, if $I$ is a monomial ideal and $u\in S$ is a monomial regular of $S/I$, then
$$\qdepth(S/I)\geq \qdepth(S/(I,u))\geq \qdepth(S/I)-1,$$
and, moreover, these inequalities are sharp. However, if $I$ is a monomial complete intersection ideal,
minimally generated by $m$ monomials, then $\hdepth(S/I)=\sdepth(S/I)=\depth(S/I)=n-m$.
In Remark \ref{noidci} we note that, in this case, we may have 
$\qdepth(I)>\sdepth(I)=n-\left\lfloor \frac{m}{2} \right\rfloor$.
We end Section $2$ with an interesting combinatorial identity. More precisely, 
in Theorem~\ref{ccuccu}, we show that if $I\subset S$ is a squarefree monomial complete intersection, minimally generated by $m$ monomials,
then $$\beta^{n-m+1}_k(S/I) + \beta^{n-m+1}_{n-m+1-k}(S/I) = 0\text{ for }0\leq k\leq n-m+1.$$
In Section $3$, as an application of our new method for computing the Hilbert depth, we tackle the case of $J_{n,m}$, 
the squarefree Veronese ideal of degree $m$, i.e., the monomial ideal generated by all the squarefree monomials of degree $m$ in $S$. 
In Proposition~\ref{kukuk} we note that $\qdepth(S/J_{n,m})=m-1$. Ge et al.\ proved in \cite{ge} that 
$$\qdepth(J_{n,m})=m+\left\lfloor \frac{n-m}{m+1} \right\rfloor.$$ 
We give an alternative proof of this result, 
which makes use of a transformation formula for hypergeometric series; see Theorem~\ref{teo3}.

  
\section{Main results}

The main result of this section and of the paper in Theorem~\ref{teo1}, where we provide a new formula
for the Hilbert depth of $J/I$, where $0\subset I\subsetneq J\subset S$ are two squarefree monomial ideals.
Using this theorem, we derive several new results, like Theorem~\ref{tu}.
First, we fix some notations. 

We denote $[n]:=\{1,2,\ldots,n\}$. 
 For a subset $C\subset [n]$, we denote $x_C:=\prod_{j\in C}x_j$.
 For two subsets $C\subset D\subset [n]$, we denote $[C,D]:=\{A\subset [n]\;:\;C\subset A\subset D\}$,
      and we call it the \emph{interval} bounded by $C$ and $D$.
 Let $I\subset J\subset S$ be two squarefree monomial ideals. We let
$$\PP_{J/I}:=\{C\subset [n]\;:\;x_C\in J\setminus I\} \subset 2^{[n]}.$$
 A partition of $\PP_{J/I}$ is a decomposition
  $$\mathcal P:\;\PP_{J/I}=\bigcup_{i=1}^r [C_i,D_i],$$
      into disjoint intervals.
If $\mathcal P$ is a partition of $\PP_{J/I}$, we let $\sdepth(\mathcal P):=\min_{i=1,\dots,r} |D_i|$.
The Stanley depth of $P_{J/I}$ is 
      $$\sdepth(P_{J/I}):=\max\{\sdepth(\mathcal P)\;:\;\mathcal P\text{ is a partition of }\PP_{J/I}\}.$$
 Herzog, Vl\u adoiu and Zheng proved in \cite{hvz} that $$\sdepth(J/I)=\sdepth(\PP_{J/I}).$$
 Let $\PP:=\PP_{J/I}$, where $I\subset J\subset S$ are squarefree monomial ideals. For any $k$ with $0\leq k\leq n$, we denote
$$\PP_k:=\{A\in \PP\;:\;|A|=k\}\text{ and }\alpha_k(J/I)=\alpha_k(\PP)=|\PP_k|.$$
 For any $q$ and $k$ with $0\leq k\leq q\leq n$, we consider:
\begin{equation}\label{betak}
\beta_k^q(J/I)=\sum_{j=0}^k (-1)^{k-j} \binom{q-j}{k-j} \alpha_{j}(J/I).
\end{equation}
Using the inverse relation \cite[Equation (7) on p. 50 with $q=0$]{RiorAA}, from \eqref{betak} we get
\begin{equation}\label{alfak}
\alpha_k(J/I)=\sum_{j=0}^k \binom{q-j}{k-j} \beta_j^q(J/I),\text{ for }0\leq k\leq q.
\end{equation}
Let $\mathcal P:\;\PP_{J/I}=\bigcup_{i=1}^r [C_i,D_i]$ be a Stanley decomposition of $\PP_{J/I}$ with
$q:=\sdepth(\mathcal P)=\sdepth(J/I)$. By refining $\mathcal P$, we can assume that if $|C_i|<q$ then $|D_i|=q$;
see \cite[Lemma~3.4]{rin}. With this assumption it is easy to see that
\begin{equation}\label{redbird}
\beta_k^q(J/I)=|\{i\;:\;|C_i|=k\}| \geq 0\text{ for all }0\leq k\leq q.
\end{equation}


\begin{lema}\label{luma}
With the above notations, for every $0\leq r\leq n$, we have that:
\begin{align*}
(1-t)^r H_{J/I}(t) & = \beta_0^r(J/I) + \beta_1^r(J/I) t + \cdots + \beta_r^r(J/I) t^r + \\
                    & + \alpha_{r+1}(J/I) \frac{t^{r+1}}{1-t} + \cdots +\alpha_n(J/I) \frac{t^{n-r}}{(1-t)^{n-r}}.
\end{align*}
\end{lema}

\begin{proof}
First, let us consider the case $J=S$. 
Let $\Delta$ be the Stanley Reisner simplicial complex associated to $I$ and assume
that $\dim(\Delta)=d-1$. We denote $f=(f_{-1},f_0,\ldots,f_{d-1})$ the $f$-vector of $\Delta$.
It is easy to see that
$$\alpha_j(S/I)=\begin{cases} f_{j-1}(\Delta),&\text{for } 0\leq j\leq
d ,\\
0,&\text{for } d+1\leq j\leq n.\end{cases}$$
It follows that
$$H_{S/I}(t)=\sum_{j=0}^d f_{j-1}(\Delta) t^j (1-t)^{-j} = \sum_{j=0}^n \alpha_j(S/I) t^j (1-t)^{-j}.$$
Multiplying the above identity with $(1-t)^r$ and using \eqref{betak} we get the required conclusion.

The general case follows from the case $J=S$ and the obvious facts:
$$ H_{J/I}(t)=H_{S/I}(t)-H_{S/J}(t),\; \alpha_j(J/I)=\alpha_j(S/I)-\alpha_j(S/J),\text{ for all }0\leq j\leq n,$$
and, therefore, $\beta_k^r(J/I)=\beta_k^r(S/I)-\beta_k^r(S/J)$ for all $0\leq k\leq r\leq n$.
\end{proof}

We recall the definition of the Hilbert depth of a module; see \cite[Definition 3.1]{uli}:

\begin{dfn}
Let $M$ be a finitely generated graded $S$-module. The \emph{Hilbert depth} of $M$ is the number
$$\hdepth(M)=\max\left\{r\;:\;
\begin{matrix}
  \text{ There exists a f.g. graded }S\text{-module }N \\
  \text{ with }H_M(t)=H_N(t)\text{ and }\depth(N)=r \end{matrix}
\right\}.$$
\end{dfn}

Furthermore, we recall the following result.

\begin{teor}(\cite[Theorem 3.2]{uli})\label{uliu}
Let $M$ be a finitely generated graded $S$-module. Then
$$\hdepth(M)=\max\{r\;:\;(1-t)^q H_M(t)\text{ is non-negative }\}.$$
\end{teor}

Now, we can prove our first main result:

\begin{teor}\label{teo1}
Let $I\subset J\subset S$ be two squarefree monomial ideals. Then:
$$\qdepth(J/I):=\max\{q\;:\;\beta_k^q(J/I) \geq 0\text{ for all }0\leq k\leq q\}.$$
\end{teor}

\begin{proof}
This follows from Lemma \ref{luma} and Theorem~\ref{uliu}.
\end{proof}

A simple, but very useful lemma is the following:

\begin{lema}\label{13}
Let $0\subset I\subsetneq J\subset S$ be two squarefree monomial ideals. Then:
\begin{enumerate}
\item[(1)] $\qdepth(J/I)\leq \max\{k\;:\;\alpha_k(J/I)>0\}$.
\item[(2)] $\qdepth(J/I)\geq \min\{k\;:\;\alpha_k(J/I)>0\}$.
\end{enumerate}
\end{lema}

\begin{proof}
(1) Let $m:=\max\{k\;:\;\alpha_k(J/I)>0\}$. If $m=n$ then there is nothing to prove. Suppose $m<n$.
From \eqref{alfak}, we have 
\begin{equation}\label{waaa}
0 = \alpha_{m+1}(J/I) = \sum_{j=0}^{m+1} \beta_j^{m+1}(J/I).
\end{equation}
If $\qdepth(J/I)\geq m+1$, from \eqref{waaa} it follows that 
$$\beta_j^{m+1}(J/I)=0\text{ for all }0\leq j\leq m+1.$$ Therefore $I=J$, a contradiction.

\medskip
(2) Here, the proof is similar.
\end{proof}

In order to extend the method of computing Hilbert depth given in Theorem~\ref{teo1} to quotients of 
arbitrary monomial ideals, we can use the well-known procedure of polarization; see for instance \cite[Section 1.5]{poli}.

Let $I\subset J\subset S$ be two monomial ideals. Let $x^g:=\lcm(u\;:\;u\in G(I)\cup G(J))$,
where $g=(g_1,\ldots,g_n)\in\mathbb N^n$ and $x^g=x_1^{g_1}\cdots x_n^{g_n}$.
We consider the polynomial ring $$R:=S[x_{ij}\;:\;1\leq i\leq n,\; 2\leq j\leq g_i].$$
Note that we added
$N=\sum\limits_{i=1}^n \max\{0,g_i-1\}$ new variables, i.e., $\dim(R)=\dim(S)+N$.

If $u=x_1^{a_1}\cdots x_n^{a_n}$ is a monomial such that $u\mid x^g$, that is, $a_i\leq g_i$ for $1\leq i\leq n$, we define the squarefree monomial 
$$u^p=x_1^{\min\{a_1,1\}}x_{12}\cdots x_{1,a_1}x_2^{\min\{a_2,1\}}x_{22}\cdots x_{2,a_2}
\cdots x_n^{\min\{a_n,1\}}x_{n2}\cdots x_{n,a_n}.$$
The polarizations of $I$ and $J$ are the squarefree monomial ideals 
$$I^p=(u^p\;:\;u\in G(I))\subset R\text{ and }J^p=(u^p\;:\;u\in G(J))\subset R.$$

\begin{prop}\label{p1}
With the above notations:
$$\qdepth(J/I):=\qdepth(J^p/I^p)-N.$$
\end{prop}

\begin{proof}
Since $J/I$ is obtained from $J^p/I^p$ by factorization with a regular sequence consisting of $N$ linear forms,
we have that $$H_{J/I}(t)=(1-t)^N H_{J^p/I^p}(t).$$
Now, the conclusion follows from Theorem~\ref{uliu}.
\end{proof}

As a direct consequence of \cite[Theorem~4.3]{ichim}, we have the following result:

\begin{prop}\label{p3}
With the above notations:
$$\sdepth(J/I)=\sdepth(J^p/I^p)-N.$$
\end{prop}

We mention the fact that the above result was generalized in \cite{ichim1} and \cite{ichim2},
where the authors show that  Stanley depth, as well as the usual depth, of $J/I$ are essentially determined 
by the so called lcm-lattices of $I$ and $J$. However, the Hilbert depth invariant is not; see \cite[Example 4.12]{ichim2}.

Now, we can reprove the following well known fact; see for instance \cite[Eq.(11), pag.38]{her}:

\begin{cor}\label{p2}
Let $I\subset J\subset S$ be two monomial ideals. Then 
$$\sdepth(J/I)\leq \qdepth(J/I).$$
\end{cor}

\begin{proof}
From Proposition \ref{p1} and Proposition \ref{p3} we can reduce to the squarefree case.
The conclusion follows from \eqref{redbird} and Theorem~\ref{teo1}.
\end{proof}

\begin{exm}\rm
Let $I=(x_1^2,x_1x_2^2)\subset S=K[x_1,x_2]$. Then $I^p=(x_1x_{12},x_1x_2x_{22})\subset R=S[x_{12},x_{22}]$.
For simplicity, we denote $x_{3}:=x_{12}$, $x_4:=x_{22}$, and thus $I^p=(x_1x_3,x_1x_2x_4)\subset R=K[x_1,x_2,x_3,x_4]$. We consider $P=P_{R/I^p}$. We denote $\alpha_j:=\alpha_j(R/I^p)$ and $\beta_k^q:=\beta_k^q(R/I^p)$ for all $j,k$ and $q$. It is easy to see that 
$$\alpha_0=1,\; \alpha_1=4,\; \alpha_2=5,\text{ and }\alpha_3=1.$$
For $q=2$, we have 
$$\beta_0^2=\alpha_0=1,\; \beta_1^2=\alpha_1-\binom{2}{1}\beta_0^2=2,\;\beta_2^2=\alpha_2-\binom{2}{2}\beta_0^2-\binom{1}{1}\beta_1^2 = 5-2-1=2.$$
For $q=3$, we have
$$\beta_0^3=\alpha_0=1,\; \beta_1^3=\alpha_1-\binom{3}{1}\beta_0^2=1,\;\beta_2^3=\alpha_2-\binom{3}{2}\beta_0^3-\binom{2}{1}\beta_1^3 = 5-3-2=0.$$
Moreover, we have $$\beta_3^3 = \alpha_3 - \binom{3}{3}\beta_0^3-\binom{2}{2}\beta_1^3 = 1-1-1=-1<0.$$
It follows that $\qdepth(R/I^p)=2$ and thus $\qdepth(S/I)=\qdepth(R/I^p)-2=0$.
\end{exm}

The following result can be seen as the counterpart of \cite[Lemma~3.6]{hvz} 
in the framework of Hilbert depth:

\begin{lema}\label{hvz}
Let $I\subsetneq J\subset S$ be two monomial ideals. Let $\overline I=I\overline S$ and $\overline J=J\overline S$ be the
extensions of $I$ and $J$ in the ring $\overline S:=S[x_{n+1}]=K[x_1,\ldots,x_{n+1}]$. Then
$$\qdepth(\overline J/\overline I)=\qdepth(J/I)+1.$$
\end{lema}

\begin{proof}
Since $x_{n+1}$ is regular on $(\overline J/\overline I)$ and
$(\overline J/\overline I)/x_{n+1}(\overline J/\overline I) \cong J/I$, we have that
$$H_{J/I}(t)=(1-t)H_{\overline J/\overline I}(t).$$
The conclusion follows from Theorem~\ref{uliu}.
\end{proof}

We also recall the following result:

\begin{prop}\label{pp21}
Let $0 \to U \to M \to N \to 0$ be a short exact sequence of finitely generated graded $S$-modules. Then:
$$\hdepth(M\oplus N)\geq \min\{\hdepth(M),\hdepth(N)\}.$$
\end{prop}

\begin{proof}
Since $H_M(t)=H_U(t)+H_N(t)$, the conclusion follows from Theorem~\ref{uliu}.
\end{proof}

As a particular case, we get:

\begin{cor}
Let $I\subsetneq J\subset S$ be two monomial ideals. Then:
\begin{enumerate}
\item[(1)] $\qdepth(S/I)\geq \min\{\qdepth(S/J),\qdepth(J/I)\}$.
\item[(2)] $\qdepth(J)\geq \min\{\qdepth(I),\qdepth(J/I)\}$.
\end{enumerate}
\end{cor}

\begin{lema}\label{pp22}
Let $I\subset J\subset S':=K[x_1,\ldots,x_m]$ be two monomial ideals and $u\in S[x_{m+1},\ldots,x_{n}]$ a monomial. Then
$$\qdepth(JS/IS)=\qdepth(u(JS/IS)).$$
\end{lema}

\begin{proof}
Since $H_{u(JS/IS)}(t)=t^{\deg(u)}H_{JS/IS}(t)$, the conclusion follows from Theorem~\ref{uliu}.
\end{proof}

As a direct consequence of Lemma~\ref{pp22} we get:

\begin{cor}\label{ppp}
Let $I\subset S$ be a monomial ideal and $u\in S$ a monomial such that $u\notin I$ and $I=u(I:u)$. Then
$\qdepth(I:u) = \qdepth(I)$.
\end{cor}

\begin{lema}
Let $I\subset S$ be a monomial ideal and $u\in S$ a monomial with $u\notin I$. Then
$$\qdepth(S/I)\geq \min\{ \qdepth(S/(I:u)),\qdepth(S/(I,u)) \}.$$
\end{lema}

\begin{proof}
Using the short exact sequence $$0\to S/(I:u) \stackrel{\cdot u}{\longrightarrow} S/I \to S/(I,u) \to 0,$$
we deduce that $$H_{S/I}(t)=t^{\deg u}\cdot H_{S/(I:u)}(t)+H_{S/(I,u)}(t).$$
The conclusion follows from Theorem~\ref{uliu}.
\end{proof}

We recall the following well-known results.

\begin{prop}\label{ccc}
  Given a monomial ideal $I\subset S$ and $u\in S$ a monomial which is not contained in $I$,
  we have the following:
\begin{enumerate}
\item[(1)] $\sdepth(S/(I:u))=\sdepth(S/I)$ if $I=u(I:u)$; see \cite[Theorem~1.1(1)]{mir};
\item[(2)] $\sdepth(S/(I:u))\geq \sdepth(S/I)$; see \cite[Proposition~1.3]{pop} {\em(}{\tt arXiv} version{\em)}; 
\item[(3)] $\sdepth(I:u)\geq \sdepth(I)$; see \cite[Proposition~2.7(2)]{mirci};
\item[(4)] $\depth(S/(I:u))\geq \depth(S/I)$; see \cite[Corollary~1.3]{asia}.
\end{enumerate}
\end{prop}

It is natural to ask if similar formulae hold for Hilbert depth. This is not the case, as the following example shows.

\begin{exm}\rm
(1) Let $I=(x_1x_2,x_2x_3,x_3x_4,x_4x_5)\subset S=K[x_1,\ldots,x_6]$. By straightforward computations, using Theorem~\ref{teo1}, we get 
$$\qdepth(S/I)=3\text{ and }\qdepth(S/x_6I)=4.$$ 
This shows that results similar to (1), (2) and (4) of Proposition~\ref{ccc} do not hold in the framework of Hilbert depth.

Let $I'=I\cap S'$, where $S'=K[x_1,\ldots,x_5]$. From Lemma~\ref{hvz}, Corollary~\ref{ppp} and the straightforward computation
of $\qdepth(I')$, we have 
$$\qdepth(I)=\qdepth(x_6I)=\qdepth(I')+1=5.$$

\medskip
(2) Let $I=(x_1x_2,x_2x_3,x_3x_4,x_4x_5,x_5x_1)$ and $J=(x_1x_2,x_2x_3,x_3x_4,x_4x_5,x_5x_1x_6)$ be ideals in
   $S=K[x_1,\ldots,x_6]$. It is clear that $(J:x_6)=I$. By straightforward computations we get
	$$\qdepth(I)=5>\qdepth(J)=4,$$
	which shows that the results (3) and (4) of Proposition~\ref{ccc} do not hold in the framework of Hilbert depth.
\end{exm}

We recall the following result.

\begin{teor}\label{cranz}
Let $I\subset S$ be a monomial ideal minimally generated by $m$ monomials. Then:
\begin{enumerate}
\item[(1)] $\sdepth(S/I)\geq n-m$; see \cite[Proposition~1.2]{mir};
\item[(2)] $\sdepth(I)\geq \max\{1,n-\lfloor \frac{m}{2} \rfloor\}$; see \cite[Theorem~2.3]{okazaki}.
\end{enumerate}
\end{teor}

\begin{cor}\label{cranza}
Let $I\subset S$ be a monomial ideal minimally generated by $m$ monomials. Then:
\begin{enumerate}
\item[(1)] $\qdepth(S/I)\geq n-m$.
\item[(2)] $\qdepth(I)\geq \max\{1,n-\lfloor \frac{m}{2} \rfloor\}$.
\end{enumerate}
\end{cor}

\begin{proof}
This follows from Theorem \ref{cranz} and Proposition~\ref{p2}.
\end{proof}

\begin{lema}\label{kkk}
Let $1\leq m < n$, and let $I'\subset S':=K[x_1,\ldots,x_m]$ be a squarefree monomial ideal. Let $u=x_{m+1}\cdots x_n$
and $s=\deg(u)=m-n$. Let $q\geq 0$ be an integer. Then
$$\beta_k^{q+s}(S/(I,u))=\begin{cases} \beta_k^q(S'/I')- \beta_{k-s}^q(S'/I'),& \text{if }0\leq k\leq q, \\
   \sum\limits_{\ell=0}^{k-q-1} \binom{k-q-1}{\ell} \alpha_{q+1+\ell}(S'/I') - \beta_{k-s}^d(S'/I'),&\text{if }q+1\leq k\leq q+s, \end{cases}  $$
where $\beta_j^q(S'/I')=0$ for $j<0$.
\end{lema}

\begin{proof}
Let $k\leq s$. Since $S/(I',u) \cong S/I'S \oplus (I,u)/I$ it follows that 
\begin{equation}\label{orbete}
\alpha_k(S/(I',u)) = \alpha_k(S/I'S) - \alpha_k((I',u)/I'S).
\end{equation}
Let $v\in S$ be a squarefree monomial. Then $v\in (I',u)\setminus I'$ if and only if $v\mid u$ and $u=vw$ with $w\in S'\setminus I'$.
It follows that $\alpha_k((I',u)/I'S)=\alpha_{k-s}(S'/I')$, and therefore from \eqref{orbete} we get
\begin{equation}\label{orbete2}
\alpha_k(S/(I',u))=\alpha_k(S/I'S)-\alpha_{k-s}(S'/I')\text{ for }0\leq k\leq d,
\end{equation}
where $\alpha_{j}(S'/I')=0$ for $j<0$. If $k<s$ then $\alpha_k(S/(I',u))=\alpha_k(S/I'S)$ and
thus $\beta_k^q(S/(I',u))=\beta_k^q(S/I'S)$. Now, assume $k\geq s$.
From \eqref{betak} it follows that
\begin{equation}\label{sunca1}
\beta_k^{q+s}(S/(I',u))=\beta_q^{d+s}(S/I'S) - \sum_{j=s}^{k}(-1)^{k-j} \binom{q+s-j}{k-j} \alpha_{j-s}(S'/I')
\text{ for  }0\leq k\leq q.
\end{equation}
Using induction on $s\geq 1$, we can easily get
\begin{equation}\label{sunca2}
\beta_k^{q+s}(S/I'S) = \begin{cases} \beta_k^q(S'/I'),&\text{for } 0\leq k\leq q, \\
   \sum\limits_{\ell=0}^{k-q-1} \binom{k-q-1}{\ell}
   \alpha_{q+1+\ell}(S'/I') ,&\text{for }q+1\leq k\leq q+s .\end{cases}
\end{equation}
On the other hand, using the substitution $m:=j-s$, we get
\begin{multline}\label{sunca3}
  \sum_{j=s}^{k}(-1)^{k-j} \binom{q+s-j}{k-j} \alpha_{j-s}(S'/I') \\
  = \sum_{m=0}^{k-s} (-1)^{(k-s)-m}\binom{q-m}{(k-s)-m}\alpha_{m}(S'/I') =
\beta^q_{k-s}(S'/I'),
\end{multline}
$\text{ for }s\leq k\leq q$.
From \eqref{sunca1}, \eqref{sunca2} and \eqref{sunca3} we get the required conclusion.
\end{proof}

We recall the following result.

\begin{teor}[{\cite[Theorem 1.1]{asia1}}]\label{asiat11}
Let $I\subset S$ be a monomial ideal and $u\in S$ a monomial,
regular of $S/I$. Then:
$$\sdepth(S/(I,u))=\sdepth(S/I)-1.$$
\end{teor}

It is natural to ask if a similar result holds in the framework of Hilbert depth.
The following result is the best we can expect.

\begin{teor}\label{tu}
Let $I\subset S$ be a monomial ideal and $u\in S$ a monomial, regular of $S/I$. Then:
\begin{enumerate}
\item[(1)] $\qdepth(S/I)\geq \qdepth(S/(I,u)) \geq \qdepth(S/I)-1$.
In particular, if $u$ is a variable then $\qdepth(S/(I,u))=\qdepth(S/I)-1$.
\item[(2)] $\qdepth((I,u))\geq \min\{\qdepth(I),\qdepth(S/I)\}$.
\end{enumerate}
\end{teor}

\begin{proof}
(1) Without loss of generality, we may assume that $I=I'S$, where $I'\subset S'=K[x_1,\ldots,x_m]$ is a squarefree
monomial ideal, and $u=x_{m+1}\cdots x_n$. Let $j:=n-m$. We use induction on $j\geq 1$. If $j=1$ then $m=n-1$, $u=x_n$ and
$$S/(I,u)\cong S/(I',x_n)\cong S'/I'.$$
Thus, the result follows from Lemma~\ref{hvz}. 

Now, assume $j\geq 2$.
For $0\leq j\leq n-m$ we let $I_j:=I+(x_{m+j+1}\cdots x_{n})\subset S$.
We have the inclusions $$I=I_0\subset I_1\subset \cdots \subset I_{n-m-1}\subset I_{n-m}=S.$$
Therefore, we have 
\begin{equation}\label{cutzu1}
S/I\cong (I_1/I_0)\oplus(I_2/I_1)\oplus \cdots \oplus (I_{n-m}/I_{n-m-1}).
\end{equation}
Since $I_{j+1}/I_j\cong S/(I,x_{m+j+1})$ for all $j$ with $0\leq j\leq n-m-1$, from Theorem~\ref{hvz} it follows that
\begin{equation}\label{cutzu2}
\qdepth(I_{j+1}/I_j)=\qdepth(S/I)-1\text{ for }0\leq j\leq n-m-1.
\end{equation}
From \eqref{cutzu1}, \eqref{cutzu2} and Proposition~\ref{pp21} it follows that
$$\qdepth(S/(I,u))\geq \qdepth(S/I)-1.$$
Thus, in order to complete the proof, we have to show the other inequality.

Let $q:=\qdepth(S'/I')$. From Lemma~\ref{kkk}, we have 
\begin{equation}\label{coocoo}
\beta^{q+1+s}_k(S/(I,u)) = \beta_k^{q+1}(S'/I') - \beta_{k-s}^{q+1}(S'/I')\text{ for }0\leq k\leq q+1.
\end{equation}
Let $j_0:=\min\{\beta_j^{q+1}(S'/I')<0\}$. From \eqref{coocoo} it follows that
$$\beta^{q+1+s}_{j_0}(S/(I,u))<0,$$
and therefore $\qdepth(S/(I,u))\leq q+s=\qdepth(S/I)$, as required.

\medskip
(2) We write $(I,u)=I\oplus (I,u)/I$. Since $(I,u)/I = u(S/I)$, the result follows from
    Proposition~\ref{pp21} and Lemma~\ref{pp22}.
\end{proof}

\begin{obs}\rm
Let $I'\subset S'=K[x_1,\ldots,x_m]$ be a squarefree monomial ideal and $u=x_{m+1}\cdots x_n$, as
in the proof of Theorem~\ref{tu}. Let $I=I'S$. 
Let $q:=\qdepth(S'/I')$ and assume that
$$\alpha_{q+1}(S'/I')<\beta^q_{q+1-\deg(u)}(S'/I').$$
From Lemma~\ref{kkk} it follows that
$$\beta^{q+s}_{q+1}(S/(I,u))=\alpha_{q+1}(S'/I')-\beta^q_{q+1-\deg(u)}(S'/I')<0,$$
hence $\qdepth(S/(I,u))\leq \qdepth(S/I)-1$. Therefore, from Theorem~\ref{tu},
it follows that $\qdepth(S/(I,u))=\qdepth(S/I)-1$. However, in general this equality does not hold, 
as the following example shows.
\end{obs}

\begin{exm}\rm
Let $I'=(x_1,x_2)\cap(x_3,x_4)\cap(x_5,x_6,x_7)\subset S'=K[x_1,x_2,\ldots,x_7]$. We also
consider the ideal $I:=I'S$ and the monomial $u=x_8x_9\in S$, where $S=K[x_1,x_2,\ldots,x_9]$.

Using CoCoA \cite{cocoa}, we computed $\qdepth(S'/I')=3$ and $\qdepth(S/(I,u))=5$.
Therefore, we have $\qdepth(S/I)=\qdepth(S/(I,u))$.
\end{exm}

Let $I\subset S$ be a squarefree monomial ideal with $G(I)=\{u_1,u_2,\ldots,u_m\}$. \

For any nonempty subset $J\subset [m]=\{1,2,\ldots,m\}$,
we let $$u_J:=\lcm(u_j\;:\;j\in J)\text{ and }d_J:=\deg(u_J).$$ 
Furthermore, we denote $u_{\emptyset}:=1$ and $d_{\emptyset}:=0$.
As usual, if $j<0$ and $m\geq 0$ we define $\binom{m}{j}:=0$.

\begin{teor}\label{includere}
With the above notations, we have 

\medskip\noindent
{\em(1)}\kern1cm  $\displaystyle
\alpha_k(I)=\sum_{\emptyset\neq J\subset [m]} (-1)^{|J|-1}
\binom{n-d_J}{k-d_J}$,

\medskip\noindent
{\em(2)}\kern1cm  $\displaystyle
\alpha_k(S/I)=\sum_{J\subset [m]} (-1)^{|J|} \binom{n-d_J}{k-d_J}$.

\medskip\noindent
{\em(3)}\kern1cm  $\displaystyle
\beta_k^q(S/I)=\sum_{J\subset [m],\;k\geq d_J} (-1)^{|J|} \binom{n-d_J-q+k-1}{k-d_J}$.
\end{teor}

\begin{proof}
(1) We use the inclusion-exclusion principle. Let $P=P_I\subset 2^{[n]}$.
We consider the subsets
\begin{equation}\label{pej}
P_j = \left\{A\in P\;:\;u_j: x_A=\textstyle\prod_{i\in A}x_i\right\}\text{ for }1\leq j\leq m.
\end{equation}
Moreover, if $J\subset [m]$ is nonempty, we denote $P_J=\bigcap_{j\in J}P_j$. In particular, $P_{\{j\}}=P_j$ for $1\leq j\leq m$.
From \eqref{pej} it follows that
\begin{equation}\label{peje}
P_J=\left\{A\in P\;:\;u_J : x_A = \textstyle\prod_{i\in A}x_i\right\}\text{ for  }\emptyset \neq J\subset [m].
\end{equation}
Since $P=P_1\cup P_2\cup \cdots \cup P_m$, it follows that
\begin{equation}\label{cardinal}
|P|=\sum_{\emptyset \neq J\subset [m]} (-1)^{|J|-1} |P_J|.
\end{equation}
On the other hand, since $\deg(u_J)=d_J$, the number of squarefree monomials of degree $k$ in $P_J$ is $\binom{n-d_J}{k-d_J}$,
that is, $\alpha_k(P_J)=\binom{n-d_J}{k-d_J}$. Therefore, the required formula follows from \eqref{peje} and \eqref{cardinal}.

\medskip
(2) This follows from (1) and the fact that 
$\alpha_k(S/I)=\binom{n}{k}-\alpha_k(I)=\binom{n-d_{\emptyset}}{k-d_{\emptyset}}-\alpha_k(I)$.

\medskip
(3) From (2) and \eqref{betak} it follows that
\begin{align*}
\beta_k^q(S/I)& =\sum_{j=0}^k (-1)^{k-j}\binom{q-j}{k-j}\sum_{J\subset [m]} (-1)^{|J|} \binom{n-d_J}{j-d_J} \\
& = \sum_{J\subset [m]} (-1)^{|J|} \sum_{j=0}^k (-1)^{k-j}\binom{q-j}{k-j} \binom{n-d_J}{j-d_J}.  
\end{align*}
If $d_J>k$ then $\sum_{j=0}^k (-1)^{k-j}\binom{q-j}{k-j} \binom{n-d_J}{j-d_J}=0$. If $d_J\leq k$ then,
using the substitution $\ell=j-d_J$ and the Chu--Vandermonde summation
(see e.g\@. \cite[Sec.~5.1, Eq.~(5.27)]{GrKPAA}), we get
\begin{align*} \sum_{j=0}^k (-1)^{k-j}\binom{q-j}{k-j} \binom{n-d_J}{j-d_J} &= \sum_{\ell=0}^{k-d_J} (-1)^{k-d_J-\ell} 
  \binom{(q-d_J)-\ell}{(k-d_J)-\ell}\binom{n-d_J}{\ell} \\ &= \binom{n-d_J-q+k-1}{k-d_J},
\end{align*}	
as required.
\end{proof}

\begin{obs}\rm
Let $0\subset I\subset J\subset S$ be two squarefree monomial ideals. From the decomposition $S/I = S/J \oplus J/I$, it follows
that $$\alpha_k(J/I)=\alpha_k(S/I)-\alpha_k(S/J)\text{ for  }0\leq k\leq n.$$
Therefore, applying Theorem~\ref{includere}, we can write $\alpha_k(J/I)$ in combinatorial terms of the degrees of the
least common multiple of the minimal monomial generators of $I$ and $J$, respectively.
\end{obs}


\begin{cor}\label{coro}
Let $I\subset S$ be a squarefree monomial complete intersection with $G(I)=\{u_1,u_2,\ldots,u_m\}$.
Let $d_j=\deg(u_j)$ for all $j$ with $1\leq j\leq n$. For a nonempty subset $J\subset [m]$ we let $d_J:=\sum_{j\in J}d_j$. Then we have

\medskip\noindent
{\em(1)}\kern1cm  $\displaystyle
\alpha_k(I)=\sum_{\emptyset\neq J\subset [m]} (-1)^{|J|-1} \binom{n-d_J}{k-d_J}$,

\medskip\noindent
{\em(2)}\kern1cm  $\displaystyle
\alpha_k(S/I)=\sum_{J\subset [m]} (-1)^{|J|} \binom{n-d_J}{k-d_J}$,

\medskip\noindent
{\em(3)}\kern1cm $\displaystyle
\beta_k^q(S/I)=\sum_{J\subset [m],\;k\geq d_J} (-1)^{|J|} \binom{n-d_J-q+k-1}{k-d_J}$.
\end{cor}

\begin{proof}
It is enough to notice that $u_J=\lcm(u_j\;:\;j\in J)=\prod_{j\in J}u_j$ and thus $d_J=\sum_{j\in J}d_j$, as required.
Subsequently we apply Theorem~\ref{includere}.
\end{proof}

We recall the following result.

\begin{teor}\label{sci}
Let $I\subset S$ be a monomial complete intersection, minimally generated by $m$ monomials. Then:
\begin{enumerate}
\item[(1)] $\sdepth(S/I)=n-m$; see Theorem~\ref{asiat11};
\item[(2)] $\sdepth(I)=n-\left\lfloor \frac{m}{2} \right\rfloor$; see \cite[Theorem~2.4]{shen}.
\end{enumerate}
\end{teor}

\begin{obs}\label{noidci}\rm
If $\mathbf m=(x_1,\ldots,x_n)$ is the maximal graded ideal of $S$, then, according to \cite[Theorem~2.2]{biro} and \cite[Example 3.4]{uli} we have
$$\qdepth(\mathbf m)=\sdepth(\mathbf m)=\left\lceil \frac{n}{2} \right\rceil.$$
Note that, according to the previous theorem, if $I$ is a monomial complete intersection, minimally generated by $m$ monomials,
then $\sdepth(I)=n- \left\lfloor \frac{m}{2} \right \rfloor$. It is natural to ask if this equality remains true if we replace
$\sdepth$ with $\hdepth$. The answer is no:

Let $I=(x_1,x_2,x_3,x_4,x_5x_6,x_7x_8)\subset S=K[x_1,\ldots,x_8]$. We have  
$$\qdepth(I)=6> 8 - \left\lfloor \frac{6}{2} \right\rfloor = 5 = \sdepth(I).$$
Hence, a result similar to Theorem~\ref{sci}(2) does not hold in the framework of Hilbert depth.
\end{obs}

Note that, according to Theorems~\ref{cranz} and~\ref{sci}, the case of monomial complete intersections $I\subset S$
gives minimal values for $\sdepth(S/I)$ and $\sdepth(I)$ in terms of the number of minimal monomial generators of $I$.
In the following proposition, we note that a similar result holds for $\qdepth(S/I)$.

\begin{prop}\label{regulat}
Let $I\subset S$ be a monomial complete intersection, minimally generated by $m$ monomials. Then
 $\qdepth(S/I)=n-m$.
\end{prop}

\begin{proof}
This result can be easily deduced from the special form of the Hilbert series of $S/I$. Here we present an alternative proof:
Without any loss of generality, we may assume that $I$ is squarefree and, moreover, that $I$ has the minimal system of generators 
$G(I)=\{u_1,\ldots,u_m\}$ with $d_j=\deg(u_j)$ for $1\leq j\leq m$. Furthermore, since $I$ is minimally generated
by the monomials $u_1,\ldots,u_m$ with disjoint supports, it is easy to see that the maximal degree of a squarefree
monomial $u$ which is not contained in $I$ is $n-m$. Therefore, we have 
\begin{equation}\label{appa}
\alpha_k(S/I)=0\text{ and }\alpha_k(I)=\binom{n}{k}\text{ for  }n-m+1\leq k\leq n.
\end{equation}
In particular, from Lemma~\ref{13}, it follows that
$$\qdepth(S/I)\leq \max\{k\;:\;\alpha_k(S/I)>0\}=n-m.$$
On the other hand, $\qdepth(S/I)\geq \sdepth(S/I)=n-m$. Hence $\qdepth(S/I)=n-m$.
\end{proof}

In the following, we use the set up of Corollary~\ref{coro}.

\begin{lema}\label{magix}
Assume that $m\geq 2$. For any integer $k$, we have that
$$\sum_{J\subset [m]} (-1)^{|J|} \binom{m-2+k-d_J}{m-2}=0.$$
\end{lema}

\begin{proof}
We use induction on $m\geq 2$ and $n\geq 2$. If $m=2$ then the conclusion is trivial, since $\binom{m-2+k-d_J}{m-2}=\binom{k-d_J}{0}=1$.
Also, $n=2$ forces $m=2$. Now, assume that $n=m\geq 3$. It follows that $d_1=d_2=\cdots=d_n=1$. Therefore, the conclusion is equivalent to
\begin{equation}\label{will}
\sum_{J\subset [n]} (-1)^{|J|} \binom{n-2+k-d_J}{n-2}= \sum_{j=0}^n (-1)^j \binom{n}{j}\binom{n-2+k-j}{n-2}=0.
\end{equation}
In order to prove \eqref{will}, we write the last sum in standard
hypergeometric notation
(cf.~\cite{SlatAC})
\begin{equation} \label{eq:hyp} 
{}_r F_s\!\left[\begin{matrix} a_1,\dots,a_r\\ b_1,\dots,b_s\end{matrix}; 
z\right]=\sum _{k=0} ^{\infty}\frac {\po{a_1}{k}\cdots\po{a_r}{k}}
{k!\,\po{b_1}{k}\cdots\po{b_s}{k}} z^k\ ,
\end{equation}
where $(\alpha)_k:=\alpha(\alpha+1)\cdots(\alpha+k-1)$, $k\ge1$, and
$(\alpha)_0:=1$. We obtain
\begin{align*}
\sum_{j=0}^n (-1)^j \binom{n}{j}\binom{n-2+k-j}{n-2}&=
\binom {n-2+k}{n-2}
{}_2 F_1\!\left[\begin{matrix} -k,-n\\2-k-n\end{matrix}; 
1\right]\\
&=
\binom {n-2+k}{n-2}\lim_{\varepsilon\to0}
{}_2 F_1\!\left[\begin{matrix}
    -k-\varepsilon,-n\\2-k-\varepsilon-n\end{matrix};  
1\right].
\end{align*}
The $_2F_1$-series can be evaluated by means of the
Chu--Vandermonde summation in hypergeometric form
(see \cite[Eq.~(1.7.7); Appendix~(III.4)]{SlatAC}),
$$                                                                             {}_2 F_1\!\left[\begin{matrix} a,-N\\c\end{matrix}; 
1\right]  = {\frac {({ \textstyle c-a}) _{N} }
    {({ \textstyle c}) _{N} }}, 
$$
where $N$ is a nonnegative integer.
Thus, we get
$$
\sum_{j=0}^n (-1)^j \binom{n}{j}\binom{n-2+k-j}{n-2}=
\binom {n-2+k}{n-2}\lim_{\varepsilon\to0}
\frac {(2-n)_n} {(2-k-\varepsilon-n)_n}.
$$
The last expression is visibly zero since
$(2-n)_n=(2-n)(3-n)\cdots (-1)\cdot 0\cdot1$.
This verifies~\eqref{will}.

Now, assume that $n>m\geq 3$. Without any loss of generality, we may assume that $d_m\geq 2$.
We consider the sequence $d_1,\ldots,d_{m-1},d_m-1$. 

For $J\subset [m]$, we denote 
$$d'_J=\begin{cases} d_J,&\text{if }m\notin J,
\\ d_J-1,&\text{if }m\in J. \end{cases}$$
Using the induction hypothesis on $n$, it follows that
$$\sum_{J\subset [m]} (-1)^{|J|} \binom{m-2+k-d'_J}{m-2} = 0.$$
Moreover, we have
$$\sum_{J\subset [m],\;m\notin J} (-1)^{|J|} \binom{m-2+k-d_J}{m-2} + 
\sum_{J\subset [m],\;m\in J} (-1)^{|J|} \binom{m-2+k-d_J+1}{m-2} = 0.$$
On the other hand
\begin{align*}
\sum_{J\subset [m],\;m\in J} (-1)^{|J|} \binom{m-2+k-d_J+1}{m-2} =& 
\sum_{J\subset [m],\;m\in J} (-1)^{|J|} \binom{m-2+k-d_J}{m-2}\\
& + \sum_{J\subset [m],\;m\in J} (-1)^{|J|} \binom{m-2+k-d_J}{m-3}.
\end{align*}
Hence, in order to complete the proof, it is enough to show that 
\begin{equation}\label{shosho}
\sum_{J\subset [m],\;m\in J} (-1)^{|J|} \binom{m-2+k-d_J}{m-3}=0.
\end{equation}
For $J\subset[m]$ with $m\in J$, we denote $J'=J\setminus\{m\}$. 
Using the induction hypothesis on $m$, it follows that 
$$ \sum_{J\subset [m],\;m\in J} (-1)^{|J|} \binom{m-2+k-d_J}{m-3}= \sum_{J'\subset[m-1]} (-1)^{|J|} \binom{m-3 + (k-d_m+1) - d_{J'}}{m-3} = 0.$$
Hence \eqref{shosho} holds, as required.
\end{proof}

\begin{teor}\label{ccuccu}
Let $I\subset S$ be a squarefree monomial complete intersection, minimally generated by $m$ monomials.
Then $$\beta^{n-m+1}_k(S/I) + \beta^{n-m+1}_{n-m+1-k}(S/I) = 0\text{ for }0\leq k\leq n-m+1.$$
\end{teor}

\begin{proof}
If $m=1$ then $I=(x_1x_2\ldots x_m)$ and it is easy to check that $\beta^n_0(S/I)=1$, $\beta^n_k(S/I)=0$, for $1\leq k\leq n-1$,
and $\beta^n_n(S/I)=-1$. Hence, we get the conclusion. Assume $m\geq 2$. From Corollary~\ref{coro}(3) it follows that
\begin{equation}\label{ttt1}
\begin{split}
& \beta^{n-m+1}_k(S/I) = 
\sum_{J\subset [m],\;k\geq d_J} (-1)^{|J|} \binom{m-2+k-d_J}{m-2},\\
& \beta^{n-m+1}_{n-m+1-k}(S/I)=\sum_{J\subset [m],\;n-m+1-k\geq d_J} (-1)^{|J|} \binom{n-k-d_J-1}{m-2}.
\end{split}
\end{equation}
For $J\subset [m]$, we denote $\overline J=[m]\setminus J$. Note that $d_{\overline J}=n-d_J$. It follows that
$$ \beta^{n-m+1}_{n-m+1-k}(S/I) = \sum_{\overline J\subset [m],\;k+m-1 \leq d_{\overline J}} (-1)^{m-|\overline J|} \binom{d_{\overline J}-k-1}{m-2}. $$
Since $\binom{d_{\overline J}-k-1}{m-2} = (-1)^{m-2} \binom{m-2-k-d_{\overline J}}{m-2}$, we therefore get
\begin{equation}\label{ttt2}
 \beta^{n-m+1}_{n-m+1-k}(S/I) = \sum_{\overline J\subset [m],\;k+m-1 \leq d_{\overline J}} (-1)^{|\overline J|} \binom{m-2+k-d_{\overline J}}{m-2}.
\end{equation}
Note that, if $d$ is an integer with $k+1\leq d\leq k+m-2$, then $\binom{m-2+k-d}{m-2}=0$. Thus, from \eqref{ttt1} and \eqref{ttt2} it follows that
$$\beta^{n-m+1}_k(S/I)+\beta^{n-m+1}_{n-m+1-k}(S/I)=\sum_{J\subset [m]} (-1)^{|J|} \binom{m-2+k-d_J}{m-2}.$$
Now, the conclusion follows from Lemma~\ref{magix}.
\end{proof}

\begin{cor}\label{coro2}
Let $I\subset S$ be a squarefree monomial complete intersection, minimally generated by $m$ monomials.
Then $$\beta_{n-m+1}^{n-m+1}(S/I) = -1.$$ 
\end{cor}

\begin{proof}
From Theorem \ref{ccuccu}, it follows that 
$$\beta_{n-m+1}^{n-m+1}(S/I) = - \beta_{0}^{n-m+1}(S/I) = - \alpha_0(S/I) = -1.$$
\end{proof}

Note that, Corollary \ref{coro2} implies that $\hdepth(S/I)\leq n-m$, which was already known, of course.

\section{Squarefree Veronese ideals}

The aim of this section is to give a new proof, based on Theorem~\ref{teo1}, for the Hilbert depth of squarefree Veronese ideals;
see Theorem~\ref{teo3}. One novelty of our approach consists in using the machinery of hypergeometric functions. Also, in Theorem~\ref{krat1}
and Theorem~\ref{krat2}, we study the nonnegativity of the coefficients $b(n,m,q,t)$ defined in \eqref{bnmqt}, which could be useful per se,
not only as a mean of proving Theorem~\ref{teo3}.

\begin{dfn}
Let $n\geq m\geq 1$ be two integers. Let $J_{n,m}$ be the ideal in $S:=K[x_1,\ldots,x_n]$ generated by all squarefree monomials
of degree $m$. 
$J_{n,m}$ is called the {\it squarefree Veronese ideal of degree} $m$.
\end{dfn}

\begin{exm}\rm
If $n=4$ and $m=2$ then there are exactly $\binom{4}{2}=6$ squarefree monomials of degree $2$ in $S:=K[x_1,x_2,x_3,x_4]$.
Moreover, we have 
$$J_{4,2}=(x_1x_2,x_1x_3,x_1x_4,x_2x_3,x_2x_4,x_3x_4)\subset S=K[x_1,x_2,x_3,x_4].$$
\end{exm}


\begin{lema}\label{jnm}
With the above notations, we have the following:
\begin{enumerate}
\item[(1)] $\alpha_k(S/J_{n,m})=\begin{cases} \binom{n}{k},&\text{if } k<m, \\ 0,& \text{if }m\leq k\leq n.\end{cases}$
\item[(2)] $\alpha_k(J_{n,m})=\begin{cases} 0,& \text{if }k<m, \\ \binom{n}{k},&\text{if } m\leq k\leq n.\end{cases}$
\end{enumerate}
\end{lema}

\begin{proof}
(1) We have $P_{S/J_{n,m}}=\{A\subset [n]\;:\;x_A=\prod_{j=1}^n x_j \notin J_{n,m}\}$.
Since $J_{n,m}$ is generated by all squarefree monomials of degree $m$, the required conclusion follows.

(2) This follows from the fact that $P_{J_{n,m}}=2^{[n]}\setminus P_{S/J_{n,m}}$ and (1).
\end{proof}

\begin{prop}\label{kukuk}
With the above notations, we have 
$$\qdepth(S/J_{n,m})=\sdepth(S/J_{n,m})=m-1.$$
\end{prop}

\begin{proof}
From \cite[Theorem 1.1(2)]{vero} it follows that $\sdepth(S/J_{n,m})=m-1$.
On the other hand, from Lemmas~\ref{13} and~\ref{jnm} it follows that $\qdepth(S/J_{n,m})\leq m-1$.
The conclusion follows from Proposition~\ref{p1}.
\end{proof}

\begin{prop}\label{proo}
We have $\qdepth(J_{n,m})\leq m+\left\lfloor \frac{n-m}{m+1} \right\rfloor$.
In particular, if $n\leq 2m$ then $\qdepth(J_{n,m})=m$.
\end{prop}

\begin{proof}
If $n=m$ then there is nothing to prove, so we may assume $n>m$.

Since $\alpha_1(J_{n,m})=\cdots=\alpha_{m-1}(J_{n,m})=0$, from Lemma~\ref{13} we get $\qdepth(J_{n,m})\geq m$.

Let $d$ be an integer with $m<d\leq n$. From Lemma~\ref{jnm} it follows that
\begin{equation}\label{cur}
 \beta_m^d(J_{n,m})=\binom{n}{m}\text{ and }\beta_{m+1}^d(J_{n,m})=\binom{n}{m+1}-(d-m)\binom{n}{m}.
\end{equation}
 If $\binom{n}{m+1}<\binom{n}{m}$, i.e., $m\geq \lfloor \frac{n+1}{2} \rfloor$, then, according to \eqref{cur},
it follows that $\beta_{m+1}^d(J_{n,m})<0$ and thus $\qdepth(J_{n,m}) = m$.

Now, assume this is not the case. From \eqref{cur} it follows that
$$\beta_{m+1}^d (J_{n,m}) < 0 \text{ if and only if } d>m+\frac{n-m}{m+1}.$$
Therefore, we get the required formula.
\end{proof}

\begin{prop}\label{proo2}
Let $m\geq 1$ and $n\geq 2m+1$ be two integers and let $q:=\left\lfloor \frac{n-m}{m+1} \right\rfloor$.
For $1\leq t\leq q$ we have 
$$\beta_{m+t}^{m+q}(J_{n,m}) = \sum_{j=0}^t (-1)^{t-j} \binom{q-j}{t-j}\binom{n}{m+j}.$$
\end{prop}

\begin{proof}
Since $n\geq 2m+1$ it follows that $q\geq 1$. Let $d:=m+\left\lfloor \frac{n-m}{m+1} \right\rfloor=m+q$.
From Lemma~\ref{jnm}(2) and \eqref{betak}
we have $\beta_m^d(J_{n,m})=\binom{n}{m}$ and 
\begin{equation}\label{pussy}
\beta_k^{m+q}(J_{n,m})=\sum_{\ell=m}^{k} (-1)^{k-\ell}\binom{m+q-\ell}{k-\ell}\binom{n}{\ell}\text{ for }m+1\leq k\leq m+q.
\end{equation}
By letting $t=k-m$ and $j=\ell-m$, the conclusion follows from \eqref{pussy}.
\end{proof}

We want to prove that
\begin{equation}\label{bnmqt}
b(n,m,q,t):=\beta_{m+t}^{m+q}(J_{n,m})=\sum_{j=0}^t(-1)^{t-j}\binom {q-j}{t-j}\binom n{m+j}
\end{equation}
is non-negative for $m,q\ge1$, $1\le t\le q$, and $n\ge mq+m+q$.
We shall show something stronger, see Theorems~\ref{krat1} and~\ref{krat2} below.


\begin{teor}\label{krat1}
Let $t$ be even. Then
$b(n,m,q,t)\ge0$ for all non-negative integers $n,m,q,t$.
\end{teor}

\begin{proof}
We do a simultaneous induction on $n$ and $m$.

For the start of the induction, we need to verify the claim
for $m=0$ and for $n=0$. Indeed, by the Chu--Vandermonde summation formula, we have
\begin{align*}
b(n,0,q,t)&=\sum_{j=0}^t(-1)^{t-j}\binom {q-j}{t-j}\binom n{j}\\
&=\sum_{j=0}^t\binom {-q+t-1}{t-j}\binom n{j}\\
&=\binom {-q+t-1+n}t.
\end{align*}
For $-q+t-1+n\ge0$ this last binomial coefficient
is non-negative, regardless whether $t$ is
even or odd. Moreover, it is also non-negative for $-q+t-1+n<0$
if $t$ is even.

On the other hand, if $t$ is even, then we have $b(0,m,q,t)=\de_{m,0}\binom qt$,
where $\de_{m,0}$ is 1 if $m=0$ and 0 otherwise. Clearly, this shows that
$b(0,m,q,t)$ is non-negative.

For the induction step, we observe that, by the standard three-term recurrence
for binomial coefficients, we have
$$
b(n,m,q,t)=b(n-1,m,q,t)+b(n-1,m-1,q,t).
$$
This relation completes the induction.
\end{proof}

\begin{lema}\label{kratl}
We have
$$
b(n,m,q,t)=\sum_{j=0}^t
(-1)^{t - j}\frac { (m + t - j - 1)!\, n!\, (n - m - q + j - 
     1)!\, (q - j)!} {(m - 1)!\,(n - m)!\,(n - m - q - 1)!\,(t - j)!\,(m + 
       t)!\,(q - t)!}.
$$
\end{lema}

\begin{proof}
We write the binomial sum $b(n,m,q,t)$ in standard hypergeometric
notation~\eqref{eq:hyp}.
We obtain
$$
b(n,m,q,t)=
(-1)^t\binom qt\binom nm
{} _{3} F _{2} \!\left [ \begin{matrix} { m-n,1,-t}\\ { -q,m+1}\end{matrix} ;
   {\displaystyle 1}\right ]  .
$$
We now apply the transformation formula (see \cite[(3.1.1)]{GaRaAA})
$$
{} _{3} F _{2} \!\left [ \begin{matrix} { a, b, -m}\\ { d, e}\end{matrix} ;
   {\displaystyle 1}\right ]  = 
  {\frac{ ({ \textstyle e-b}) _{m} }{({ \textstyle e}) _{m} }}\,
{} _{3} F _{2} \!\left [ \begin{matrix} { -m, b, d-a}\\ { d, 1 + b - e
       - m}\end{matrix} ; { 1}\right ],
$$
where $m$ is a non-negative integer.
After some manipulation, one obtains the claimed expression.
\end{proof}

\begin{teor}\label{krat2}
Let $t$ be odd. Then
$b(n,m,q,t)\ge0$ for all non-negative integers $n,m,q,t$
with
\begin{equation}  
n\ge\max\{mq+m + q -(t-1)(m+1),mq+m +q-\tfrac {q(m-1)(t-1)}t\}.
\label{eq:CK1}
\end{equation}
\end{teor}

\begin{proof}
Clearly, $b(n,m,q,0)=\binom nm\ge0$. Furthermore, 
for $q<t$, we have
$$b(n,m,q,t)=\sum_{j=q+1}^t(-1)^{t-j}\binom {q-j}{t-j}\binom n{m+j}
=\sum_{j=q+1}^t\binom {t-q-1}{t-j}\binom n{m+j}>0.
$$
Hence, we may assume $t\ge1$ and $q\ge t$ from
now on.

We use the expression from the lemma.
We investigate the growth properties of the summand
(without sign)
$$
f(n,m,q,t,j)=\frac { (m + t - j - 1)!\, n!\, (n - m - q + j - 
     1)!\, (q - j)!} {(m - 1)!\,(n - m)!\,(n - m - q - 1)!\,(t - j)!\,(m + 
       t)!\,(q - t)!}.
$$
We claim that, under the assumption \eqref{eq:CK1}, $f(n,m,q,t,j)$ is
(weakly) monotone
increasing in $j$. Indeed, we have
$$
\frac {f(n,m,q,t,j)} {f(n,m,q,t,j-1)}
=\frac {(n- m  - q+j-1) (t-j+1)} {(q-j+1) (m +t-j)}.
$$
This will be at least 1 if
\begin{equation} 
n\ge 1 - 2 j + 2 m + 2 q + \frac {(q-t)(m-1) } {t-j+1}.
\label{eq:CK2}
\end{equation}
We consider the right-hand side as a function in $j$.
Differentiation of the right-hand side yields
$$
- 2  + \frac {(q-t)(m-1) } {(t-j+1)^2}.
$$
Equating this to zero, we arrive at a quadratic equation in $j$ with
two real solutions, one less than $t+1$, one greater than $t+1$
(here we use that $t\ge1$ and $q\ge t$).
Since the function (in $j$) on the right-hand side of \eqref{eq:CK2} tends to $+\infty$
as $j\to-\infty$ and as $j\to (t+1)^-$,
it is convex for $j\in[1,t]$.
The maximum on $[1,t]$ is therefore found as the greater of the values of the
right-hand side of \eqref{eq:CK2} at the boundary points $j=1$ and $j=t$.
Indeed, these two values are the ones of which the maximum is taken on the
right-hand side of \eqref{eq:CK1}.

Now we pair the summands in the sum that defines
$b(n,m,q,t)$, 
\begin{multline*} 
b(n,m,q,t)=\big(f(n,m,q,t,t)-f(n,m,q,t,t-1)\big)\\
+\big(f(n,m,q,t,t-2)-f(n,m,q,t,t-3)\big)+\cdots.
\end{multline*}
Given the just proved monotonicity of the summand, we see that each
pair produces a non-negative value, proving non-negativity of
$b(n,m,q,t)$ under the assumption \eqref{eq:CK1}.
\end{proof}

Now, we can reprove the main result of \cite{ge}.

\begin{teor}(\cite[Theorem 1.2]{ge}) \label{teo3}
Let $n\geq m\geq 1$. We have that:
$$\qdepth(J_{n,m})=m+\left\lfloor \frac{n-m}{m+1} \right\rfloor.$$
\end{teor}

\begin{proof}
Let $q=\left\lfloor \frac{n-m}{m+1} \right\rfloor$. From Proposition~\ref{proo} we have
$\qdepth(J_{n,m})\leq m+q$. The other inequality follows from Theorem~\ref{krat1}, Theorem~\ref{krat2}, \eqref{bnmqt}
and Proposition~\ref{proo2}.
\end{proof}

\begin{obs}\label{ultima_obs}\rm
In \cite{vero} we proposed the conjecture that $\sdepth(J_{n,m})=m+\left\lfloor \frac{n-m}{m+1} \right\rfloor$; see \cite[Conjecture 1.6]{vero},
and we proved that it holds for $n\leq 3m$; see \cite[Theorem~1.1]{vero} and \cite[Corollary~1.5]{vero}. Keller et al.\ 
improved this result to $n\leq 5m+3$; see \cite[Theorem~1.1]{keller}. Note that $n\leq 5m+3$ means $q\leq 3$.
In Theorem~\ref{teo3} we reproved that $\qdepth(J_{n,m})=m+\left\lfloor \frac{n-m}{m+1} \right\rfloor$ for any
$n\geq m\geq 1$. However, this does not imply that a similar result holds for $\sdepth$, 
but it shows that such a result is credible.
\end{obs}

\subsection*{Acknowledgements}

We gratefully acknowledge the use of the computer algebra system Cocoa (cf. \cite{cocoa}) for our experiments.

Mircea Cimpoea\c s was supported by a grant of the Ministry of Research, Innovation and Digitization, CNCS - UEFISCDI, 
project number PN-III-P1-1.1-TE-2021-1633, within PNCDI III.

\subsection*{Data availability}

Data sharing not applicable to this article as no data sets were generated or analyzed
during the current study.

\subsection*{Conflict of interest}

The authors have no relevant financial or non-financial interests to disclose.


\end{document}